\documentclass[a4paper,twoside]{article}

\usepackage{hyperref}   
\usepackage{amsrefs}

\usepackage{amsmath,amsthm,amsfonts,latexsym,amscd,amssymb,enumerate}

\theoremstyle{plain}
\newtheorem{theorem}{Theorem}

\theoremstyle{definition}

\theoremstyle{remark}

\newcommand{\reals}{\mathbb{R}}

\newcommand{\abs}[1]{\left\lvert#1\right\rvert} 

\DeclareMathOperator{\im}{im}      

\DeclareMathOperator{\spec}{Spec}

\DeclareMathOperator{\ind}{ind}

\DeclareMathOperator{\scal}{scal}  

\newcommand{\forget}[1]{}

\newcommand{\innerprod}[1]{\langle #1 \rangle}

{\catcode`@=11\global\let\c@equation=\c@theorem}

\allowdisplaybreaks[2]

\begin{document}
\pagestyle{myheadings}
\markboth{Thomas Schick}{Invertibility of Dirac
with Hilbert-$A$-module coefficients}

\title{A note on invertibility of the Dirac operator twisted
with Hilbert-$A$-module coefficients}

\author{ Thomas Schick\thanks{
\protect\href{mailto:thomas.schick@math.uni-goettingen.de}{e-mail:
  thomas.schick@math.uni-goettingen.de}
\protect\\
\protect\href{http://www.uni-math.gwdg.de/schick}{www:~http://www.uni-math.gwdg.de/schick}}\\
Mathematisches Institut\\
Georg-August-Universit\"at G{\"o}ttingen\\
Germany}
\maketitle

\begin{abstract}
  Given a closed connected spin manifold $M$ with non-negative scalar
  curvature which is non-zero, we show that the Dirac operator twisted with
  any flat Hilbert module bundle is invertible.
\end{abstract}

Let $M$ be a compact spin manifold with Riemannian metric $g$. It is an
important and standard fact that the spectrum of
the spin Dirac 
operator is restricted by the scalar curvature of $M$. By the
Schr\"odinger-Lichnerowicz formula, if
$\scal_g(x)\ge 4c^2$ for every $x\in M$ then
\begin{equation*}
spec(D)\cap (-c,c)=\emptyset.
\end{equation*}

This is a direct consequence of the Schr\"odinger-Lichnerowicz formula
\begin{equation*}
  D^2 = \nabla^*\nabla + \frac{\scal_g}{4},
\end{equation*}
  where the connection Laplacian $\nabla^*\nabla$ is a non-negative operator,
  and the operator of multiplication by $\frac{\scal_g}{4}$ is bounded below by
  $c^2$.

  Note that this argument works exactly the same way if we replace  the
  classical Atiyah-Singer Dirac operator $D$ by $D_E$, where we twist with a
  bundle $E$ with flat connection, because the Schr\"odinger-Lichnerowicz
  formula remains unchanged. In particular, this also works for twists with a
  bundle $E$ of $Hilbert$-$A$-modules with a $C^*$-algebra $A$, such that
  $D_E$ is an operator in the Mishchenko-Fomenko calculus. This is important
  for higher index theory, then $\ind(D_E)\in K_*(A)$, and the invertibility
  implies of course that $\ind(D_E)=0\in K_*(A)$, with its applications to the
  topology of positive scalar curvature, compare e.g.~\cite{SchickICM}.

There is another easy generalization of the above spectral consideration which
is well known: if $M$ is connected and $\scal_g(x)\ge 0$, with
$\scal_g(x_0)>0$ for some $x_0\in M$, then we can argue as follows:

For the usual untwisted Dirac operator $D$ take $s\in\ker(D)$. Then
\begin{equation*}
0=\innerprod{Ds,Ds}_{L^2} = \innerprod{\nabla S,\nabla
  s}_{L^2} + \int_{M} \frac{\scal_g(x)}{4} \innerprod{s(x),s(x)}_x \,dvol_g.
\end{equation*}

The assumptions then imply that $\nabla s=0$, i.e.~$s$ is parallel and in
particular $\innerprod{s(x),s(x)}_x$ is constant. Because by assumption
$\scal_g(x)>0$ for $x$ in a neighborhood of $x_0$, this implies $s=0$.
Consequently, $\ker(D)=0$

As $D$ has discrete spectrum, this implies again that $D$ is invertible,
i.e.
\begin{equation*}
spec(D)\cap (-\epsilon,\epsilon)=\emptyset\text{ for some} \epsilon>0.
\end{equation*}

In this note, we now prove that this extends to all operators $D_E$ for flat
Hilbert-$A$-module bundles $E$, even though in general
the spectrum of $D_E$ is not discrete if $\dim(A)=\infty$. In that case, the
argument just given 
does not work. It only gives $\ker(D_E)=\{0\}$, which does not imply that
$D_E$ is invertible.

This question came up in the analysis of the geometry of the space of metrics
of non-negative scalar curvature carried out recently in \cite{SchickWraith}.

\begin{theorem}
  Let $M$ be a connected closed spin manifold with Riemannian metric $g$, $A$ a $C^*$-algebra and $E\to
  M$ a flat Hilbert $A$-module bundle over $M$. Assume that $\scal_g(x)\ge 0$
  for all $x\in M$, and $\scal_g(x_0)>0$ for some $x_0\in M$.

  Then $D_E$ is invertible, i.e.~there is $c_0>0$ such that $\spec(D_E)\cap
  (-c_0,c_0)=\emptyset$. 
\end{theorem}
\begin{proof}
  Instead of arguing with $s\in\ker(D_E)$, consider for $ c>0$ a function
  $f_c\colon \reals\to[0,1]$ with $f_c(x)>0$ if and only if $\abs{x}<c$ and
  then $f_c(D_E)$. This is a replacement for the spectral
  projector $P_c:=\chi_{[-c,c]}(D_E)$ which in general one can't build because
  the Hilbert $A$-module morphisms don't form a von Neumann algebra. We will
  show that for $c$ sufficiently
  small, $\im(f_c(D_E))=\{0\}$. This implies that $f_c(D_E)=0$ and, by the
  choice of $f_c$, that $\spec(D_E)\cap (-c,c)=\emptyset$. By contraposition,
  assume that $s\in\im(f_c(D_E))$ with 
  $|s|_{L^2}\ne 0$. We will show that this implies that $c>c_0$ for some $c_0>0$
  depending on the geometry of $(M,g)$.

  In the following, the norms and inner products have values in $A$, and the
  inequalities $a\le b$ refer to the partial order in $A$.

  First, we have $|D_E^ks|_{L^2} \le c^k$ for all $k>0$.

  By the Sobolev embedding theorem, $s$ is smooth and there are a priori
  estimates (depending on the geometry of $M$) for the supremum norm of $s$ and
  all its covariant derivatives.

  Next, the Schr\"odinger-Lichnerowicz formula implies
  \begin{equation*}
    c^2\abs{s}_{L^2}^2 \ge \abs{D_Es}^2_{L^2} = \abs{\nabla s}^2_{L^2} + \int_{M} \frac{\scal_g(x)}{4}\innerprod{s(x),s(x)}_x\,dvol_g.
  \end{equation*}
Because $\scal_g(x)\ge 0$
  for all $x\in M$ this implies
  \begin{equation}\label{eq:SL}
    \abs{\nabla s}^2_{L^2} \le c^2\abs{s}_{L^2}^2\quad\text{and} \qquad \int_{M}
    \frac{\scal_g(x)}{4}\abs{s(x)}^2_x\,dvol_g \le c^2\abs{s}_{L^2}^2.
  \end{equation}
 Choose an small open disk $U=B_{2r}(x_0)$ around $x_0$ such that $\scal_g(x)\ge 4a>0$ for $x\in
 U$ (note that $U$ and $a$ depend only on $g$).

 We obtain then from \eqref{eq:SL}
 \begin{equation}
   \label{eq:local_est}
   \int_U \abs{s(x)}^2_x \le \int_U \frac{\scal_g(x)}{4a} \abs{s(x)}^2_x \le
   \int_M \frac{\scal_g(x)}{4a}\abs{s(x)}_x^2 \le \frac{c^2}{a} \abs{s}^2_{L^2}.
 \end{equation}

Choose a smooth cutoff function $\phi\colon M\to [0,1]$ which is equal to $1$
outside $B_{2r}(x_0)$ and vanishes on $B_r(x_0)$ and consider $\tilde
s:=\phi\cdot s$. Note that there is $C_g>0$ depending only on the geometry of
$M$ such that
\begin{equation}\label{eq:Cg}
\abs{\nabla \phi(x)}^2\le C_g\quad\forall x\in U,\qquad \nabla\phi(x) =0
\forall x\in M\setminus U.
\end{equation}


 Because $M$ is connected, we can apply the Poincar\'e inequality
 \cite[Proposition 5.2]{Taylor} (extended to
 sections of  Hilbert $A$-module
 bundles) for the
 subset $B_r(x_0)$ inside $M$:
  \begin{equation}\label{eq:Poin_applied}
   \abs{ \tilde s}^2_{L^2} \le C_{g,U} \abs{\nabla \tilde s}^2_{L^2} =
   C_{g,U}\abs{(\nabla\phi) s+\phi\nabla s}^2_{L^2} \le 2C_{g,U}(\abs{(\nabla \phi )s}_{L^2}^2 +
   \abs{\phi\nabla s}_{L^2}^2)
 \end{equation}
with constant $C_{g,U}$ depending only on the set $U$ and the geometry $g$ of
$M$.

Note finally that
\begin{equation}\label{eq:s_tilde}
  \abs{\tilde s}^2_{L^2} = \int_M \abs{s(x)}^2_x + \int_U (\phi(x)^2-1)
  \abs{s(x)}_x^2 \ge \abs{s}^2_{L^2} - \int_U \abs{s(x)}^2_x
  \stackrel{~\eqref{eq:local_est}}{\ge} (1-\frac{c^2}{a}) \abs{s}^2_{L^2},
\end{equation}
while on the other hand
\begin{equation}
  \label{eq:nabla_phi_s}
  \abs{(\nabla \phi)s}^2_{L^2} = \int_M \abs{\nabla\phi(x)}^2\abs{s(x)}_x^2
  \stackrel{~\eqref{eq:Cg}}{\le} \int_U C_g \abs{s(x)}^2_x
  \stackrel{~\eqref{eq:local_est}}{\le} \frac{C_g}{a} c^2 \abs{s}^2_{L^2}.
\end{equation}

Combining \eqref{eq:Poin_applied} with \eqref{eq:s_tilde},
\eqref{eq:nabla_phi_s}, and \eqref{eq:SL} we finally obtain
\begin{equation*}
  (1-\frac{c^2}{a})\abs{s}^2_{L^2} \le \abs{\tilde s}^2_{L^2} \le 2 C_{g,U}
  (\frac{C_g}{a}c^2 +c^2) \abs{s}^2_{L^2}.
\end{equation*}
As $s\ne 0$ and therefore $\abs{s}^2_{L^2}>0$ in $A$, this implies that $c$
must be sufficiently large, explicitly
\begin{equation*}
c> c_0:= \left( \frac{2C_{g,U} C_g +1}{a} + 2 C_{g,U}\right)^{-1/2}.
\end{equation*}

  Now, the constants $C_{g,U}$, $C_g$, and $a$ depend only on the geometry of
  $M$ 
  and the assertion follows.
\end{proof}

\noindent\textbf{Acknowledgements.} Thanks to Rudolf Zeidler for useful
comments and pointing out inaccuracies in a previous version of this note.

\bibliography{Invertible.bib}

\end{document}